\documentclass[12pt]{amsart}

\usepackage {amsmath}
\usepackage[applemac]{inputenc}
\usepackage {amsfonts, amssymb}
\usepackage {latexsym}

\newtheorem{theorem}{Theorem}

\theoremstyle{definition}

\newcommand {\R} {\mathbb{R}}
\def\be{\begin{equation}}
\def\ee{\end{equation}}

\title{A remark on approximation of open sets with regular bounded ones}

\author[Dmitry Vorotnikov]{}


\email{mitvorot@mat.uc.pt}

\begin{document}
\maketitle

% Enter the first author's name and address:
\centerline{\scshape Dmitry Vorotnikov}
\medskip
{\footnotesize
% please put the address of the first author
 \centerline{CMUC, Department of
Mathematics,}
   \centerline{University of Coimbra, Apartado 3008,}
   \centerline{3001-454 Coimbra, Portugal}
} % Do not forget to end the {\footnotesize by the sign }

\bigskip

% The name of the associate editor will be entered by an editorial staff
 %\centerline{(Communicated by the associate editor name)}

%The abstract of your paper
\begin{abstract} We show that any open set in $\R^n$ is a union of an ascending sequence of bounded open sets with analytic boundary. This is just a technical result, which is probably known. We believe, however, that it can be useful for studing BVPs on irregular open sets.
\end{abstract}

%\section {Introduction}

A boundary of a \emph{domain} (this word means any open set in $\R^n$) is called \emph{analytic} if it is an analytic manifold and the domain is locally located on one side of it.

\begin{theorem} Any domain $\Omega$ is a union of an ascending sequence of bounded domains $\Omega_m$ with analytic boundary. Moreover, $\overline{\Omega}_m\subset \Omega$. \end{theorem} 

\begin{proof} a) If $\Omega$ is bounded and connected, then the statement of the theorem is a direct consequence of \cite[Lemma 1]{amick} (see also \cite[Section XI.14]{kel}).

b) Let $\Omega$ be any bounded domain. Then it is a union of at most countable number of open connected components $\omega_m$. Each of them is a union of an ascending sequence $\omega_{m,k}$ of bounded domains with analytic boundary, and $\overline{\omega}_{m,k}\subset \omega_m$.
Now, the sequence $\Omega_m=\{ \bigcup_{l=1}^m \omega_{l,m}\}$ proves the claim. Observe that for any compact set $V\subset \Omega$ there exists $k=k(V,\Omega,\{\Omega_m\})$ such that $V \subset \Omega_k$. 
%If not, the compact sets $ V\backslash \Omega_k,\ k\in \mathbb{N},$ would have a common point, which is impossible ($\cup\Omega_k=\Omega$). 

c) Let $\Omega$ be an unbounded domain, and let $\Omega(m)$ be the intersections of $\Omega$ with the open balls of radii $m$ centered at the origin. Let $\omega(m)$ denote the set of points $x$ of $\Omega(m)$ such that the distance from $x$ to $\partial \Omega(m)$ is larger than or equal to $1/m$.   Every $\Omega(m)$ is a union of an ascending sequence $\Omega_{m,k}$ of bounded domains with analytic boundary. Then the required sequence $\Omega_m=\Omega_{m,k_m}$ is determined by the recurrence relation 
$$k_1=k(\omega(1), \Omega(1), \{\Omega_{1,k}\}),$$ 
$$k_m=k(\omega(m)\cup \overline\Omega_{m-1}, \Omega(m), \{\Omega_{m,k}\}).$$ 
\end{proof}

\end{document}